\documentclass[10pt,a4paper]{amsart}

\usepackage[T1]{fontenc}
\usepackage[utf8]{inputenc}
\usepackage{lmodern}
\usepackage{microtype}
\usepackage[a4paper,textwidth=124mm,textheight=186mm,centering]{geometry}
\usepackage{amsmath,amssymb,amsthm,mathtools,mathrsfs}
\usepackage{enumitem}
\usepackage[hidelinks]{hyperref}

\newtheorem{theorem}{Theorem}[section]
\newtheorem{proposition}[theorem]{Proposition}
\newtheorem{lemma}[theorem]{Lemma}
\newtheorem{corollary}[theorem]{Corollary}
\theoremstyle{definition}
\newtheorem{definition}[theorem]{Definition}
\newtheorem{example}[theorem]{Example}
\theoremstyle{remark}
\newtheorem{remark}[theorem]{Remark}

\newcommand{\Mod}{\mathsf{Mod}}

\newcommand{\Hom}{\operatorname{Hom}}

\newcommand{\Ext}{\operatorname{Ext}}
\newcommand{\Soc}{\operatorname{Soc}}
\newcommand{\im}{\operatorname{Im}}

\newcommand{\ess}{\leq_{\mathrm e}}
\newcommand{\summand}{\leq^{\oplus}}

\title[Smith-orbit classification for $C4^{\ast}$-modules]
{Smith-Orbit Classification of Extensions and Exact-Sequence Asymmetry for
$C4^{\ast}$-Modules}

\author[C. Gokavarapu]{Chandrasekhar Gokavarapu}
\address{Department of Mathematics\\
Government College (Autonomous)\\
Rajahmundry, Andhra Pradesh 533105, India\\
ORCID: \href{https://orcid.org/0009-0006-5306-371X}{0009-0006-5306-371X}}
\email{chandrasekhargokavarapu@gmail.com}
\email{chandrasekhargokavarapu@gcrjy.ac.in}
\thanks{Corresponding author: Chandrasekhar Gokavarapu.}


\hypersetup{
  pdftitle={Smith-Orbit Classification of Extensions and Exact-Sequence Asymmetry for C4*-Modules},
  pdfauthor={Chandrasekhar Gokavarapu},
  pdfsubject={Module theory},
  pdfkeywords={C4-module, C4*-module, exact sequence, Dedekind domain, DVR, Smith profile, Ext}
}

\begin{document}

\begin{abstract}
Let $R$ be a ring and let a $C4^{\ast}$-module mean a module all of whose
submodules are $C4$-modules.  We first determine its asymmetric
exact-sequence behavior: kernel closure is unconditional, while explicit
commutative examples disprove split-extension and cokernel closure.  We then
classify the finite-length torsion objects over a Dedekind domain.  A primary
component is $C4$ exactly when it is homocyclic, and it is $C4^{\ast}$ exactly
when it is cyclic or semisimple; in finite length $C4^{\ast}$ and strongly
$C4^{\ast}$ coincide.

The main result is a complete orbit calculation over a discrete valuation
ring $V$.  Put $A=(V/(\pi^a))^r$, $C=(V/(\pi^c))^s$ and
$m=\min\{a,c\}$.  Then
\[
 \Ext^1_V(C,A)\cong
 \operatorname{Mat}_{r\times s}(V/(\pi^m)),
\]
and the $\operatorname{Aut}(A)\times\operatorname{Aut}(C)$-orbits are
classified by a Smith profile
$0\leq\nu_1\leq\cdots\leq\nu_q\leq m$,
$q=\min\{r,s\}$.  The corresponding middle term is
\[
 \bigoplus_{i=1}^{q}
 \bigl(V/(\pi^{\nu_i})\oplus V/(\pi^{a+c-\nu_i})\bigr)
 \oplus (V/(\pi^a))^{r-q}\oplus(V/(\pi^c))^{s-q},
\]
where $V/(\pi^0)=0$.  This gives exact orbit-level criteria: the middle
term is $C4$ precisely for the zero orbit with $a=c$, or for an invertible
square orbit with $r=s$; it is $C4^{\ast}$ precisely for the zero orbit with
$a=c=1$, or for a unit scalar orbit with $r=s=1$.  We reformulate the Smith
data as a valuation--rank profile, determine the complete geometry and
additivity of the preservation loci, and globalize the result prime by prime.
Consequently every extension between arbitrary finite-length torsion
$C4^{\ast}$-modules over a Dedekind domain is decided by explicit local orbit
invariants.
\end{abstract}

\maketitle
\markboth{C. Gokavarapu}
{Smith-orbit classification for $C4^{\ast}$-modules}

\section{Introduction}

The $C4$ condition was introduced to weaken the classical $C3$ condition
without abandoning the summand geometry that makes the latter useful.  A
module $M$ is $C4$ when, for every decomposition $M=A\oplus B$ and every
homomorphism $f:A\to B$ whose kernel is a direct summand of $A$, the image
of $f$ is a direct summand of $B$.  The theory and its equivalent
perspectivity formulations were developed in
\cite{DingIbrahimYousifZhou2017C4,AltunOzarslanIbrahimOzcanYousif2018Perspective},
with the dual $D4$ theory in \cite{DingIbrahimYousifZhou2017D4}.  Later work
connected $C4$-modules with exchange and decomposition phenomena
\cite{IbrahimYousif2019DualSquareFree,IbrahimYousif2022C4Exchange,Nielsen2010SquareFreeExchange}.
This connection belongs to the wider exchange-module theory developed, in
particular, by Warfield and by Zimmermann-Huisgen--Zimmermann
\cite{Warfield1972Exchange,ZimmermannHuisgenZimmermann1984Exchange}; those
results motivate the summand viewpoint but do not decide the hereditary or
exact-sequence questions considered here.

The hereditary enlargement is recent.  Following
\cite{IbrahimEidElGuindy2026OnC4}, a module is $C4^{\ast}$ when every one of
its submodules is $C4$.  The same source defines the strongly
$C4^{\ast}$-modules by adding the semi-weak-CS condition.  These definitions
lead naturally to an exact-sequence question.  Given
\[
0\longrightarrow A\longrightarrow B\longrightarrow C\longrightarrow0,
\]
which implications among the $C4^{\ast}$ properties of $A,B,C$ are valid?
The question looks symmetric, but the universal quantifier over submodules
already predicts otherwise.  Kernels are submodules; cokernels are
quotients; and the middle term of an extension may contain mixed submodules
that occur in neither end term.

There are two classical orbit theories adjacent to this question.  The first
studies automorphism orbits inside a fixed finite abelian group.  It begins
with Birkhoff's subgroup analysis \cite{Birkhoff1935Subgroups}; explicit
automorphism descriptions appear in
\cite{HillarRhea2007Automorphisms,GolasinskiGoncalves2008Automorphisms}, and
element, tuple, pair and subgroup orbits are developed in
\cite{SchwachhoferStroppel1999Orbits,DuttaPrasad2011Degenerations,
CalvertDuttaPrasad2013Tuples,AnilkumarPrasad2014Pairs}.  Those actions keep an
ambient module fixed.  Our acting group instead changes the identifications
of both endpoints of a Yoneda extension: it is
$\operatorname{Aut}(A)\times\operatorname{Aut}(C)$ acting on
$\operatorname{Ext}^1(C,A)$.

The second neighboring theory classifies subgroup embeddings and counts
extensions from their partition data.  Hall's algebra of partitions and
Green's finite general linear group theory initiated the Hall-polynomial
framework \cite{Hall1959AlgebraPartitions,Green1955Characters}, while Klein
connected extensions of finite $p$-modules with Schur functions and Hall
polynomials \cite{Klein1968SchurExtensions,Klein1969HallPolynomial}; see also
Macdonald's systematic account \cite{Macdonald1995SymmetricFunctions}.
Littlewood--Richardson sequences encode existence and degeneration data for
subgroup embeddings \cite{Schmidmeier2007LRSequences,Schmidmeier2011LRTableau,
KosakowskaSchmidmeier2017Box}.  The representation theory of bounded
submodules and invariant subspaces shows how intricate the unrestricted
embedding category can become
\cite{Schmidmeier2005Bounded,RingelSchmidmeier2008Invariant}, and automorphisms
of short exact sequences recover Hall-polynomial information
\cite{Schmidmeier2012HallAutomorphisms}.  Recent work treats fixed elementary
subgroups or quotients over discrete valuation rings
\cite{KosakowskaSchmidmeierSchreiner2025Elementary}.

Primary-component detection for finitely generated torsion modules was
previously obtained in the localization framework of
\cite{Gokavarapu2026Localization}.  Homocyclic primary groups also appear
in the classification of Utumi abelian groups
\cite{CalugareanuDas2023Utumi}.  Here the primary reduction is an input,
not the endpoint.  We first distinguish the invariant-factor loci for $C4$
and $C4^{\ast}$.  We then solve the orbit problem for every extension between
homocyclic primary $C4$-modules, not only for cyclic or semisimple end terms.
Thus the contribution is not a substitute for Hall--Green--Klein theory:
for fixed homocyclic endpoints it gives a closed Smith-form classification
of endpoint-automorphism orbits and then identifies exactly which orbits
preserve $C4$ or $C4^{\ast}$, information not determined by the three
endpoint partitions alone.

The formal asymmetry is only the first layer.  The deeper problem is to
recognize, from computable invariants of an extension, exactly when the
middle term crosses the $C4$ and $C4^{\ast}$ boundaries.  We solve the
classification problem for finite-length torsion modules over a Dedekind
domain and the full extension-orbit problem for homocyclic modules over its
local discrete valuation rings.  If
\[
A=(V/(\pi^a))^r,\qquad C=(V/(\pi^c))^s,\qquad
m=\min\{a,c\},
\]
then an extension class is an $r\times s$ matrix over the chain ring
$V/(\pi^m)$.  Endpoint automorphisms act by invertible row and column
operations.  Smith reduction therefore produces a complete orbit invariant,
but the invariant is used here in a second, less formal way: each Smith entry
determines two elementary divisors of the middle term.  This converts the
orbit problem into an exact $C4/C4^{\ast}$ phase diagram.

To expose the information carried by a Smith form without choosing bases, we
introduce the valuation--rank profile
\[
\rho_h(X)=
\dim_{V/(\pi)}\bigl(\operatorname{Im}(X\bmod\pi^h)/
\pi\operatorname{Im}(X\bmod\pi^h)\bigr),\qquad 1\leq h\leq m.
\]
The profile records the number of Smith exponents below every level $h$ and
hence is a complete orbit invariant.  Its first level is ordinary matrix rank
when $m=1$, while in the scalar case its unique jump is the valuation.  Thus
the previously separate rank and valuation pictures are boundary shadows of
one filtered invariant.

The principal results are as follows.

\begin{enumerate}[label=\textup{(\arabic*)}]
\item The $C4^{\ast}$ class is hereditary.  Consequently, the kernel of
every homomorphism whose domain is $C4^{\ast}$ is again $C4^{\ast}$
(Theorem~\ref{thm:kernel}).  No hypothesis on the codomain or on an exact
structure is required.
\item The class is not closed under finite direct sums, hence not under
split extensions.  The split sequence with middle term $k[t]^2$ gives the
counterexample, and the same example works for strongly $C4^{\ast}$
(Theorem~\ref{thm:split-failure}).
\item The class is not closed under quotients or cokernels.  For
$D=k[t,x,y]$ and $I=(x,y)^2$, the modules $I$ and $D$ are strongly
$C4^{\ast}$, but $D/I$ is not $C4^{\ast}$
(Theorem~\ref{thm:quotient-failure}).
\item A finite decomposition is called submodule-separating when every
submodule decomposes along its components.  Under this independently
checkable condition, both $C4^{\ast}$ and strongly $C4^{\ast}$ are
componentwise properties (Theorems~\ref{thm:separating-c4star} and
\ref{thm:separating-strong}).  Central idempotents make the condition
automatic, yielding factorwise detection over finite product rings
(Theorem~\ref{thm:product}).
\item Every finite-length module is semi-weak-CS.  Over a Dedekind domain,
the primary invariant factors then give a complete classification:
$C4$ is equivalent to homocyclicity of every primary component, whereas
$C4^{\ast}$ and strongly $C4^{\ast}$ are equivalent to every primary
component being cyclic or semisimple
(Theorem~\ref{thm:pid-classification}).
\item For homocyclic end terms over a discrete valuation ring,
$\Ext^1$ is a matrix space over a truncated valuation ring.  The endpoint
automorphism orbits are classified by Smith profiles, equivalently by the
valuation--rank profile (Theorem~\ref{thm:smith-orbits}).  An explicit direct
sum formula gives the middle term of every orbit
(Theorem~\ref{thm:smith-middle}).
\item The middle term is $C4$ precisely on two sharply defined orbits: the
zero orbit when the endpoint exponents agree, and the invertible square orbit
when the endpoint ranks agree.  It is $C4^{\ast}$ precisely on the
semisimple split orbit or the cyclic unit orbit.  These statements constitute
a complete local phase diagram for extensions between all finite-length
primary $C4$-modules (Theorem~\ref{thm:complete-local-phase}).
\item The preservation loci are therefore exactly one of
$\varnothing$, $\{0\}$, a general linear group, or the union of the last two;
their additive behavior is classified without exceptions
(Corollary~\ref{cor:locus-geometry}).  Primary decomposition globalizes the
orbit criteria to every extension between finite-length torsion
$C4^{\ast}$-modules over a Dedekind domain
(Theorem~\ref{thm:global-phase}).
\end{enumerate}

The arguments combine direct-summand calculus, primary decomposition, Baer's
criterion, the natural automorphism action on $\Ext^1$, and elementary-divisor
theory over principal ideal rings.  We use the latter in the classical sense
of Kaplansky \cite{Kaplansky1949Elementary}; standard module-theoretic and
homological background is taken from
\cite{AndersonFuller1992RingsCategories,Lam1999Lectures,
Wisbauer1991Foundations,Rotman2009Homological,Weibel1994HomologicalAlgebra}.

\section{Definitions and hereditary consequences}

Throughout, $R$ is an associative ring with identity and modules are
unitary right $R$-modules.  We write $N\summand M$ when $N$ is a direct
summand of $M$, and $N\ess M$ when $N$ is essential in $M$.  Standard facts
about kernels, quotients, composition length and exact sequences are used
as in \cite{Rotman2009Homological,Weibel1994HomologicalAlgebra}.

\begin{definition}[\cite{DingIbrahimYousifZhou2017C4}]
An $R$-module $M$ is a \emph{$C4$-module} if, whenever
$M=A\oplus B$ and $f:A\to B$ is an $R$-homomorphism with
$\ker f\summand A$, one has $\im f\summand B$.
\end{definition}

The kernel formulation is important.  Replacing ``$\im f$ is a direct
summand'' by ``$\im f$ is isomorphic to a direct summand'' produces a
different statement and will not be used here.

\begin{definition}[\cite{IbrahimEidElGuindy2026OnC4}]
An $R$-module $M$ is a \emph{$C4^{\ast}$-module} if every submodule of $M$
is a $C4$-module.
\end{definition}

We recall the strong condition in enough detail to make the later arguments
self-contained.  Let $\mathscr S(M)$ consist of triples $(X,Y,\alpha)$ such
that $X,Y\summand M$ are semisimple, $X\cap Y=0$, and
$\alpha:X\to Y$ is an isomorphism.  Put
$(X,Y,\alpha)\leq(X',Y',\alpha')$ when $X\leq X'$, $Y\leq Y'$, and
$\alpha'|_X=\alpha$.

\begin{definition}[\cite{IbrahimEidElGuindy2026OnC4}]
The module $M$ is \emph{semi-weak-CS} if, for every chain
$\{(X_\lambda,Y_\lambda,\alpha_\lambda)\}_{\lambda\in\Lambda}$ in
$\mathscr S(M)$, the unions
\[
X=\bigcup_{\lambda}X_\lambda,
\qquad
Y=\bigcup_{\lambda}Y_\lambda
\]
admit direct summands $A,B\summand M$ with $X\ess A$ and $Y\ess B$.
It is \emph{strongly $C4^{\ast}$} if it is both $C4^{\ast}$ and
semi-weak-CS.
\end{definition}

The zero module is allowed as a direct summand.  In particular, the chain
containing only $(0,0,0)$ is realized by $A=B=0$.

\begin{lemma}\label{lem:invariance}
The properties $C4$, $C4^{\ast}$, semi-weak-CS and strongly
$C4^{\ast}$ are invariant under module isomorphism.
\end{lemma}

\begin{proof}
An isomorphism transports direct-sum decompositions, kernels, images and
splittings.  It also gives an inclusion-preserving bijection on submodules.
For the semi-weak-CS condition it transports semisimple direct summands,
chains in $\mathscr S(M)$ and essential embeddings.  Each definition is
therefore preserved and reflected.
\end{proof}

\begin{proposition}[Hereditary principle]\label{prop:hereditary}
If $M$ is $C4^{\ast}$ and $N\leq M$, then $N$ is $C4^{\ast}$.
Consequently, direct summands and submodules of $C4^{\ast}$-modules are
$C4^{\ast}$.
\end{proposition}

\begin{proof}
Let $L\leq N$.  Then $L\leq M$, so $L$ is $C4$ by the definition of
$C4^{\ast}$ for $M$.  This holds for every $L\leq N$.
\end{proof}

This one-line proof is the reason that kernel stability requires no transfer
package.

The following obstruction, proved in the foundational paper in an
equivalent form, is the rigidity input for the arithmetic classification.

\begin{proposition}[Disjoint-embedding obstruction]
\label{prop:embedding-obstruction}
Let $M$ be $C4^{\ast}$ and let $X,Y\leq M$ satisfy $X\cap Y=0$.
If there is a monomorphism $u:X\to Y$, then $X$ is semisimple.  In
particular, a $C4^{\ast}$-module cannot contain two disjoint isomorphic
nonsimple submodules.
\end{proposition}

\begin{proof}
This is Proposition~3.4 of \cite{IbrahimEidElGuindy2026OnC4}; we include a
proof from the defining $C4$ test.  Let $K\leq X$.  The submodule
$K\oplus Y$ is $C4$, and the monomorphism $u|_K:K\to Y$ has zero kernel.
Consequently $u(K)\summand Y$; write $Y=u(K)\oplus Y_0$.  The submodule
\[
X\oplus Y=u(K)\oplus(X\oplus Y_0)
\]
is also $C4$.  Apply its $C4$ condition to
$u^{-1}:u(K)\to X\oplus Y_0$.  The image is $K$ and the kernel is zero, so
$K\summand X\oplus Y_0$.  Restricting a retraction onto $K$ to $X$ gives
$K\summand X$.  Since every submodule $K$ of $X$ is a direct summand,
$X$ is semisimple.  The last assertion follows by taking an isomorphism
$X\to Y$.
\end{proof}

\begin{proposition}\label{prop:socle-zero}
If $\Soc(M)=0$, then $M$ is semi-weak-CS.  Hence a zero-socle
$C4^{\ast}$-module is strongly $C4^{\ast}$.
\end{proposition}

\begin{proof}
A nonzero semisimple submodule contains a simple submodule and therefore
meets $\Soc(M)$ nontrivially.  Thus $M$ has no nonzero semisimple direct
summand.  Every triple in $\mathscr S(M)$ is $(0,0,0)$, and the defining
condition holds with $A=B=0$.
\end{proof}

\begin{lemma}\label{lem:domain}
Let $D$ be a commutative domain.  Then $D_D$ is $C4^{\ast}$.  If $D$ is
not a field, then $D_D$ is strongly $C4^{\ast}$.
\end{lemma}

\begin{proof}
Every nonzero ideal $J$ of $D$ is torsion-free of rank one.  If
$J=J_1\oplus J_2$ with $J_1,J_2\neq0$, tensoring with the fraction field
would give
$1=\operatorname{rank}J=\operatorname{rank}J_1+
\operatorname{rank}J_2\geq2$, a contradiction.  Hence every nonzero
submodule of $D_D$ is indecomposable.  An indecomposable module is $C4$:
its only decompositions have a zero component.  Thus every ideal of $D$ is
$C4$, proving that $D_D$ is $C4^{\ast}$.

If $D$ is not a field, it has no nonzero simple ideal.  Indeed, multiplication
identifies every nonzero principal ideal with $D_D$, so a simple nonzero
ideal would force $D_D$ itself to be simple and $D$ to be a field.  Hence
$\Soc(D_D)=0$, and Proposition~\ref{prop:socle-zero} applies.
\end{proof}

\section{The exact-sequence asymmetry}

We now settle the three elementary closure directions.  The positive kernel
statement is formal; the negative extension and quotient statements require
explicit modules.

\begin{theorem}[Unconditional kernel stability]\label{thm:kernel}
Let $f:M\to N$ be any homomorphism.  If $M$ is $C4^{\ast}$, then
$\ker f$ is $C4^{\ast}$.  Equivalently, in every short exact sequence
\[
0\longrightarrow A\longrightarrow B\longrightarrow C\longrightarrow0,
\]
the implication
\[
B\text{ is }C4^{\ast}\quad\Longrightarrow\quad
A\text{ is }C4^{\ast}
\]
holds without a condition on $C$.
\end{theorem}

\begin{proof}
The kernel is a submodule of $M$.  Apply
Proposition~\ref{prop:hereditary}.
\end{proof}

\begin{corollary}\label{cor:strong-kernel}
If $M$ is $C4^{\ast}$ with $\Soc(M)=0$, then the kernel of every
homomorphism out of $M$ is strongly $C4^{\ast}$.
\end{corollary}

\begin{proof}
Let $K\leq M$ be a kernel.  Theorem~\ref{thm:kernel} gives that $K$ is
$C4^{\ast}$, while $\Soc(K)\leq\Soc(M)=0$.  Apply
Proposition~\ref{prop:socle-zero}.
\end{proof}

The converse direction for a split sequence is also immediate when the
middle term is already known to be $C4^{\ast}$.

\begin{proposition}\label{prop:split-quotient}
If $M$ is $C4^{\ast}$ and $N\summand M$, then $M/N$ is
$C4^{\ast}$.
\end{proposition}

\begin{proof}
Choose $L$ with $M=N\oplus L$.  The quotient $M/N$ is isomorphic to
$L$, and $L$ is $C4^{\ast}$ by Proposition~\ref{prop:hereditary}.
\end{proof}

It is tempting to reverse Proposition~\ref{prop:split-quotient} and infer
that a split extension of two $C4^{\ast}$-modules is $C4^{\ast}$.  The
following elementary example rules this out.

\begin{theorem}[Failure under split extensions]\label{thm:split-failure}
Let $k$ be a field and $D=k[t]$.  In the split exact sequence
\[
0\longrightarrow D\xrightarrow{a\mapsto(a,0)}D^2
\xrightarrow{(a,b)\mapsto b}D\longrightarrow0,
\]
both end terms are strongly $C4^{\ast}$, whereas the middle term is not a
$C4$-module.  Consequently:
\begin{enumerate}[label=\textup{(\arabic*)}]
\item $C4^{\ast}$-modules are not closed under binary direct sums;
\item $C4^{\ast}$-modules are not closed under split extensions;
\item strongly $C4^{\ast}$-modules are not closed under binary direct sums
or split extensions.
\end{enumerate}
\end{theorem}

\begin{proof}
By Lemma~\ref{lem:domain}, $D_D$ is strongly $C4^{\ast}$.  Write
$D^2=A\oplus B$, where $A=De_1$ and $B=De_2$.  Define
\[
f:A\longrightarrow B,
\qquad f(e_1)=te_2.
\]
The map is injective, so $\ker f=0\summand A$.  Its image is
$tD e_2$.  Direct summands of the rank-one free $D$-module $B$ correspond
to idempotents of $D$.  Since the domain $D$ has only the idempotents $0$
and $1$, the proper nonzero submodule $tD e_2$ is not a direct summand of
$B$.  Thus $D^2$ fails the defining $C4$ condition.
\end{proof}

The quotient direction fails even when the kernel and the middle term are
strongly $C4^{\ast}$.

\begin{theorem}[Failure under quotients and cokernels]
\label{thm:quotient-failure}
Let $k$ be a field,
\[
D=k[t,x,y],\qquad I=(x,y)^2,
\qquad S=D/I.
\]
In the exact sequence
\[
0\longrightarrow I\longrightarrow D\longrightarrow S\longrightarrow0,
\]
the modules $I$ and $D$ are strongly $C4^{\ast}$, but $S$ is not
$C4^{\ast}$.  Hence the $C4^{\ast}$ and strongly $C4^{\ast}$ classes are
not closed under quotients or cokernels of monomorphisms.
\end{theorem}

\begin{proof}
The polynomial ring $D$ is a nonfield domain, so $D_D$ is strongly
$C4^{\ast}$ by Lemma~\ref{lem:domain}.  The ideal $I$ is $C4^{\ast}$ by
Proposition~\ref{prop:hereditary}.  Moreover,
$\Soc(I)\leq\Soc(D)=0$, so $I$ is strongly $C4^{\ast}$ by
Proposition~\ref{prop:socle-zero}.

Let bars denote residue classes in $S$ and put
\[
J=(x,y)/I=S\bar x\oplus S\bar y.
\]
As a $k[t]$-module, $J$ is free with basis $\bar x,\bar y$; the elements
$\bar x$ and $\bar y$ act trivially on $J$.  Define the $S$-homomorphism
\[
g:S\bar x\longrightarrow S\bar y,
\qquad g(\bar x)=t\bar y.
\]
Under the identifications $S\bar x\cong k[t]\cong S\bar y$, this is
multiplication by $t$.  It is injective, so its kernel is a direct summand,
but its image $tS\bar y$ is a proper nonzero submodule of the rank-one
$k[t]$-module $S\bar y$ and is not a direct summand.  Therefore $J$ is not
$C4$.  Since $J\leq S$, the module $S$ is not $C4^{\ast}$.
\end{proof}

\begin{remark}\label{rem:closure-table}
Theorems~\ref{thm:kernel}, \ref{thm:split-failure} and
\ref{thm:quotient-failure} give the complete elementary closure table:
kernel closure holds unconditionally; extension closure fails even for
split sequences; and cokernel closure fails for an inclusion between
strongly $C4^{\ast}$-modules.  Thus the three directions cannot be packaged
as a symmetric ``two-out-of-three'' principle.
\end{remark}

\section{Submodule-separating decompositions}

The split counterexample fails because $D^2$ has mixed submodules and mixed
homomorphisms between its two coordinates.  We now isolate a condition that
excludes precisely this phenomenon.

\begin{definition}
A finite decomposition
\[
M=M_1\oplus\cdots\oplus M_n
\]
is \emph{submodule-separating} if every submodule $N\leq M$ satisfies
\[
N=(N\cap M_1)\oplus\cdots\oplus(N\cap M_n).
\tag{4.1}\label{eq:separation}
\]
\end{definition}

The condition is imposed on submodules, not on a chosen family of test
maps.  It can therefore be verified independently of the $C4$ conclusion.

\begin{remark}\label{rem:separation-restrictive}
Submodule separation is a restrictive hypothesis.  It is not a generic
property of direct sums: Lemma~\ref{lem:cross-hom} shows that it forces every
homomorphism between submodules of distinct components to vanish.  Its main
natural source here is a decomposition induced by finitely many orthogonal
central idempotents, as treated in Section~5.  This strength is precisely what
excludes the mixed-submodule obstruction used in Theorem~\ref{thm:split-failure}.
\end{remark}

\begin{lemma}\label{lem:cross-hom}
If $M=\bigoplus_{i=1}^n M_i$ is submodule-separating, then
\[
\Hom_R(X_i,X_j)=0
\]
for all $i\neq j$ and all submodules $X_i\leq M_i$, $X_j\leq M_j$.
Every homomorphism between submodules of $M$ is consequently componentwise.
\end{lemma}

\begin{proof}
Let $h:X_i\to X_j$ with $i\neq j$.  Its graph
\[
\Gamma_h=\{x+h(x):x\in X_i\}
\]
is a submodule of $M_i\oplus M_j\leq M$.  By
\eqref{eq:separation}, it is the direct sum of its intersections with the
$M_r$.  The intersection with $M_j$ is zero, while the intersection with
$M_i$ corresponds to $\ker h$.  Hence $\Gamma_h=\ker h$, forcing $h=0$.

Now let $X,Y\leq M$ and $f:X\to Y$.  Both $X$ and $Y$ separate by
\eqref{eq:separation}; the off-diagonal components of $f$ vanish by the
first part.
\end{proof}

\begin{lemma}\label{lem:summand-component}
Under a submodule-separating decomposition, a submodule
$N=\bigoplus_iN_i$ is a direct summand of a submodule
$L=\bigoplus_iL_i$ if and only if $N_i\summand L_i$ for every $i$.
\end{lemma}

\begin{proof}
If $L=N\oplus K$, then $K$ separates and
$L_i=N_i\oplus K_i$ after intersecting with $M_i$.  The converse follows by
taking the direct sum of componentwise complements.
\end{proof}

\begin{theorem}[Componentwise $C4^{\ast}$ criterion]
\label{thm:separating-c4star}
Let $M=\bigoplus_{i=1}^nM_i$ be a submodule-separating decomposition.
Then
\[
M\text{ is }C4^{\ast}
\quad\Longleftrightarrow\quad
M_i\text{ is }C4^{\ast}\text{ for every }i.
\]
\end{theorem}

\begin{proof}
The forward implication is Proposition~\ref{prop:hereditary}.  Conversely,
assume that all $M_i$ are $C4^{\ast}$ and take $N\leq M$.  Write
$N=\bigoplus_iN_i$ with $N_i=N\cap M_i$.  To verify that $N$ is $C4$,
let
\[
N=A\oplus B,
\qquad f:A\longrightarrow B,
\qquad \ker f\summand A.
\]
All submodules involved separate.  Thus
$A=\bigoplus_iA_i$, $B=\bigoplus_iB_i$, and, by
Lemma~\ref{lem:cross-hom}, $f=\bigoplus_if_i$ with
$f_i:A_i\to B_i$.  Lemma~\ref{lem:summand-component} gives
$\ker f_i\summand A_i$ for every $i$.  Since $N_i\leq M_i$ and $M_i$ is
$C4^{\ast}$, the module $N_i$ is $C4$; hence
$\im f_i\summand B_i$.  A second application of
Lemma~\ref{lem:summand-component} yields
$\im f=\bigoplus_i\im f_i\summand B$.  Thus $N$ is $C4$.  As $N$ was
arbitrary, $M$ is $C4^{\ast}$.
\end{proof}

We next treat the semi-weak-CS layer.  For a finite direct sum, essentiality
is componentwise: $\bigoplus_iX_i\ess\bigoplus_iA_i$ if and only if
$X_i\ess A_i$ for all $i$.  This follows by testing intersections with
submodules supported in one component.

\begin{theorem}[Componentwise strong criterion]
\label{thm:separating-strong}
Let $M=\bigoplus_{i=1}^nM_i$ be submodule-separating.  Then $M$ is
semi-weak-CS if and only if every $M_i$ is semi-weak-CS.  Consequently,
\[
M\text{ is strongly }C4^{\ast}
\quad\Longleftrightarrow\quad
M_i\text{ is strongly }C4^{\ast}\text{ for every }i.
\]
\end{theorem}

\begin{proof}
Assume first that every $M_i$ is semi-weak-CS and let
$\{(X_\lambda,Y_\lambda,\alpha_\lambda)\}$ be a chain in
$\mathscr S(M)$.  Submodule separation and
Lemma~\ref{lem:cross-hom} decompose this chain componentwise into chains
\[
\{(X_{\lambda i},Y_{\lambda i},\alpha_{\lambda i})\}
\subseteq\mathscr S(M_i).
\]
For each $i$, choose direct summands $A_i,B_i\summand M_i$ realizing the
unions.  Then $A=\bigoplus_iA_i$ and $B=\bigoplus_iB_i$ are direct
summands of $M$, and componentwise essentiality shows that they realize the
original unions.  Hence $M$ is semi-weak-CS.

Conversely, fix $i$ and a chain in $\mathscr S(M_i)$.  Pad all its terms
with zero components to obtain a chain in $\mathscr S(M)$.  A realization
in $M$ separates componentwise, and its $i$th components realize the
original chain.  Thus $M_i$ is semi-weak-CS.  Combine this equivalence with
Theorem~\ref{thm:separating-c4star}.
\end{proof}

\begin{corollary}[Controlled split transfer]\label{cor:controlled-split}
Let
\[
0\longrightarrow A\longrightarrow B\longrightarrow C\longrightarrow0
\]
split, and identify $B=A\oplus C$.  If this decomposition is
submodule-separating, then $B$ is $C4^{\ast}$ (respectively strongly
$C4^{\ast}$) if and only if both $A$ and $C$ have that property.
\end{corollary}

\begin{proof}
Apply Theorem~\ref{thm:separating-c4star} or
Theorem~\ref{thm:separating-strong}.
\end{proof}

The hypothesis fails in Theorem~\ref{thm:split-failure}: the submodule
generated by $e_1+te_2$ is not the sum of its intersections with the two
coordinate copies of $k[t]$.

\section{Finite product rings}

Submodule separation is automatic when the components are cut out by
central idempotents.  This turns the preceding abstract criterion into a
large concrete family.

\begin{lemma}\label{lem:central-separation}
Let $1=e_1+\cdots+e_n$ be a decomposition into pairwise orthogonal central
idempotents of $R$.  For every $M\in\Mod\text{-}R$,
\[
M=Me_1\oplus\cdots\oplus Me_n
\]
is submodule-separating.
\end{lemma}

\begin{proof}
For $m\in M$, one has $m=\sum_i me_i$.  If $N\leq M$ and $m\in N$,
then $me_i\in N$ because $N$ is an $R$-submodule.  Thus
$N=\bigoplus_iNe_i$, and $Ne_i=N\cap Me_i$.
\end{proof}

\begin{theorem}[Finite product theorem]\label{thm:product}
Let $R=R_1\times\cdots\times R_n$, and let $M_i=Me_i$ be the associated
$R_i$-module components of an $R$-module $M$.  Then:
\begin{enumerate}[label=\textup{(\arabic*)}]
\item $M$ is $C4$ if and only if every $M_i$ is $C4$;
\item $M$ is $C4^{\ast}$ if and only if every $M_i$ is $C4^{\ast}$;
\item $M$ is semi-weak-CS if and only if every $M_i$ is semi-weak-CS;
\item $M$ is strongly $C4^{\ast}$ if and only if every $M_i$ is strongly
$C4^{\ast}$.
\end{enumerate}
\end{theorem}

\begin{proof}
Lemma~\ref{lem:central-separation} and
Theorems~\ref{thm:separating-c4star} and
\ref{thm:separating-strong} prove (2)--(4).  For (1), every decomposition,
homomorphism, kernel and image in $\Mod$-$R$ is componentwise under the
standard equivalence
\[
\Mod\text{-}R\simeq
\prod_{i=1}^n\Mod\text{-}R_i.
\]
Direct-summand status is also componentwise, so the defining $C4$ test is
equivalent to its $n$ component tests.
\end{proof}

\begin{corollary}\label{cor:product-class}
Under the standard equivalence of module categories, the full subcategory
of $C4^{\ast}$ $R$-modules corresponds to the product of the full
subcategories of $C4^{\ast}$ $R_i$-modules.  The analogous statement holds
for strongly $C4^{\ast}$-modules.
\end{corollary}

\begin{proof}
Theorem~\ref{thm:product} identifies the objects, and the usual product-ring
equivalence identifies all morphisms.
\end{proof}

\begin{example}\label{ex:product-domains}
Let $D_1,\dots,D_n$ be commutative nonfield domains and
$R=\prod_iD_i$.  By Lemma~\ref{lem:domain} and
Theorem~\ref{thm:product}, the regular module $R_R$ is strongly
$C4^{\ast}$.  More generally, any product module whose $i$th component is
uniserial is strongly $C4^{\ast}$ by Theorem~\ref{thm:uniserial} below.
\end{example}

The product theorem is stronger than closure under a particular split
sequence: it describes every submodule and every $C4$ test through central
blocks.  Its hypothesis is supplied here by the product-ring action itself,
so no general exchange or decomposition theorem is being invoked.

\section{Finite-length classification over Dedekind domains}

We now pass from isolated examples to a complete invariant-factor
classification.  The first observation removes the additional chain
condition from every finite-length problem.

\begin{proposition}\label{prop:finite-length-semiweak}
Every finite-length module is semi-weak-CS.  Consequently, for a
finite-length module $M$,
\[
M\text{ is strongly }C4^{\ast}
\quad\Longleftrightarrow\quad
M\text{ is }C4^{\ast}.
\]
\end{proposition}

\begin{proof}
Let $\{(X_\lambda,Y_\lambda,\alpha_\lambda)\}$ be a chain in
$\mathscr S(M)$.  The two ascending chains $\{X_\lambda\}$ and
$\{Y_\lambda\}$ stabilize because $M$ is noetherian.  Since the triples
are linearly ordered, both components stabilize simultaneously from some
term onward, and compatibility then fixes the isomorphism as well.  The
two unions are therefore the semisimple direct summands occurring in that
term.  Taking them as their own essential extensions proves that $M$ is
semi-weak-CS.  The equivalence follows from the definition of strongly
$C4^{\ast}$.
\end{proof}

We shall also use the following lattice-theoretic source of cyclic examples.

\begin{theorem}\label{thm:uniserial}
Every uniserial module is strongly $C4^{\ast}$.
\end{theorem}

\begin{proof}
Every submodule of a uniserial module is uniserial and hence
indecomposable.  An indecomposable module is $C4$, because each of its
direct-sum decompositions has a zero component.  Thus the original module
is $C4^{\ast}$.  Its only direct summands are $0$ and itself; consequently
it has no pair of disjoint isomorphic nonzero semisimple direct summands.
Every chain in $\mathscr S(M)$ is therefore the zero chain, proving the
semi-weak-CS condition.
\end{proof}

Let $D$ be a Dedekind domain.  For a finite-length torsion $D$-module $M$,
write $M_{(\mathfrak p)}$ for its $\mathfrak p$-primary component.  Only
finitely many of these components are nonzero.

\begin{lemma}[Primary separation]\label{lem:primary-separation}
The primary decomposition
\[
M=\bigoplus_{\mathfrak p}M_{(\mathfrak p)}
\]
is submodule-separating.  Hence the $C4$, $C4^{\ast}$,
semi-weak-CS and strongly $C4^{\ast}$ properties of $M$ are detected
componentwise.
\end{lemma}

\begin{proof}
Choose a nonzero annihilator ideal
$\mathfrak a=\prod_{i=1}^t\mathfrak p_i^{n_i}$ of $M$.  The Chinese
remainder theorem gives pairwise orthogonal central idempotents in
\[
D/\mathfrak a\cong\prod_{i=1}^tD/\mathfrak p_i^{n_i}.
\]
Their actions on $M$ cut out the primary components.  Every $D$-submodule
is a $D/\mathfrak a$-submodule and is therefore stable under these
idempotents.
Lemma~\ref{lem:central-separation} and Theorem~\ref{thm:product} apply.
This is the finite-length case of the primary-component criterion in
\cite[Theorem~5.4]{Gokavarapu2026Localization}; the idempotent proof is
included to fix the precise decomposition used below.
\end{proof}

\begin{theorem}[Invariant-factor classification]
\label{thm:pid-classification}
Let $D$ be a Dedekind domain and let $M$ be a finite-length torsion
$D$-module.  Write
\[
M_{(\mathfrak p)}\cong
\bigoplus_{j=1}^{r_{\mathfrak p}}
D/\mathfrak p^{\lambda_{\mathfrak p j}},
\qquad
\lambda_{\mathfrak p1}\geq\cdots\geq
\lambda_{\mathfrak p r_{\mathfrak p}}\geq1.
\tag{6.1}\label{eq:elementary-divisors}
\]
Then the following hold.
\begin{enumerate}[label=\textup{(\arabic*)}]
\item $M$ is $C4$ if and only if, for every $\mathfrak p$, all the
$\lambda_{\mathfrak p j}$ are equal; equivalently, every primary component is
homocyclic.
\item $M$ is $C4^{\ast}$ if and only if, for every $\mathfrak p$, either
$r_{\mathfrak p}=1$ or
$\lambda_{\mathfrak p1}=\cdots=
\lambda_{\mathfrak p r_{\mathfrak p}}=1$; equivalently, every primary
component is cyclic or semisimple.
\item $M$ is strongly $C4^{\ast}$ if and only if it is $C4^{\ast}$.
\end{enumerate}
\end{theorem}

\begin{proof}
The displayed decomposition is the standard structure theorem for
finite-length torsion modules over a Dedekind domain
\cite{AndersonFuller1992RingsCategories,Lam1999Lectures,
FuchsSalce2001NonNoetherian}.  By Lemma~\ref{lem:primary-separation}, it is
enough to fix $\mathfrak p$.  Localization identifies the category of
finite-length $\mathfrak p$-primary $D$-modules with the corresponding
torsion category over the discrete valuation ring $D_{\mathfrak p}$;
submodules, homomorphisms and direct summands are unchanged.  Choose a
uniformizer $p$ of $D_{\mathfrak p}$ and write
\[
M_{(\mathfrak p)}\cong
\bigoplus_{j=1}^rD_{\mathfrak p}/(p^{\lambda_j}).
\]

Suppose first that $M_{(\mathfrak p)}$ is $C4$ and that two exponents
satisfy $a>b$.  Use a summand
$A\cong D_{\mathfrak p}/(p^b)$ as the first term of a decomposition and
put all remaining cyclic summands, including a summand
$V\cong D_{\mathfrak p}/(p^a)$, in the second term $B$.  Multiplication by
$p^{a-b}$ gives a monomorphism $A\to V\leq B$.  Its image is a proper
nonzero submodule of the uniserial module $V$, hence is not a direct
summand of $V$.  If it were a direct summand of $B$, a retraction restricted
to $V$ would make it a direct summand of $V$.  This contradicts the $C4$
test.  Thus all exponents must be equal.

Conversely, suppose
$M_{(\mathfrak p)}\cong(D_{\mathfrak p}/(p^n))^r$ and set
$R_n=D_{\mathfrak p}/(p^n)$.  We first verify that $R_n$ is self-injective.
Every ideal of $R_n$ has the form $p^jR_n$.  If
$h:p^jR_n\to R_n$ is a homomorphism and $h(p^j)=z$, then
$p^{n-j}z=0$, so $z\in p^jR_n$.  Hence $h$ is multiplication by an
element of $R_n$ and extends to $R_n$.  Baer's criterion proves
self-injectivity.  Therefore the finite free $R_n$-module
$M_{(\mathfrak p)}$ is injective as an $R_n$-module.

We record why this implies the required $C4$ condition.  Given
$M_{(\mathfrak p)}=A\oplus B$ and $f:A\to B$ with
$A=\ker f\oplus A_1$, put $I=f(A_1)$.  The inverse
$I\to A_1$ of $f|_{A_1}$ extends, by injectivity, to an endomorphism $h$
of $M_{(\mathfrak p)}$.  If
$\pi_1:M_{(\mathfrak p)}\to A_1$ is the projection, then
\[
f|_{A_1}\,\pi_1h:M_{(\mathfrak p)}\longrightarrow I
\]
restricts to the identity on $I$.  Thus $I=\im f$ is a direct summand of
$M_{(\mathfrak p)}$, and restriction of the retraction to $B$ makes
$I\summand B$.  This proves (1).

If $M_{(\mathfrak p)}$ is $C4^{\ast}$, then it is $C4$, so (1) makes it
homocyclic, say $(D_{\mathfrak p}/(p^n))^r$.  When $r\geq2$ and $n\geq2$,
two coordinate summands are disjoint, isomorphic and nonsimple,
contradicting Proposition~\ref{prop:embedding-obstruction}.  Hence $r=1$
or $n=1$.  Conversely, a cyclic primary module over a Dedekind domain is
uniserial and a module annihilated by $\mathfrak p$ is semisimple; in
either case it is $C4^{\ast}$.  This proves (2), again using primary
separation.  Finally, (3) follows from
Proposition~\ref{prop:finite-length-semiweak}.
\end{proof}

\begin{corollary}[Finite abelian groups]\label{cor:finite-abelian}
Let $G$ be a finite abelian group, with Sylow decomposition
\[
G_p\cong\bigoplus_{j=1}^{r_p}\mathbb Z/p^{\lambda_{pj}}\mathbb Z.
\]
Then $G$ is a $C4$ $\mathbb Z$-module if and only if every $G_p$ is
homocyclic.  It is $C4^{\ast}$ if and only if every $G_p$ is cyclic or
elementary abelian, and in this case it is automatically strongly
$C4^{\ast}$.
\end{corollary}

\begin{proof}
Apply Theorem~\ref{thm:pid-classification} with $D=\mathbb Z$.
\end{proof}

The gap between (1) and (2) is structural: a homocyclic block is
self-injective and therefore $C4$, but two nonsimple coordinate copies
create exactly the disjoint-embedding obstruction to hereditary $C4$.

\section{Smith profiles for homocyclic extensions}

The invariant-factor theorem identifies the finite-length primary $C4$
modules: they are precisely the homocyclic modules.  We now classify every
extension between two such modules.  This is strictly broader than the
cyclic-by-cyclic and semisimple-by-semisimple cases, and it is the local
calculation from which the global preservation theorem will follow.
Equivalently, a short exact sequence here is a subgroup embedding with
homocyclic kernel and quotient.  In contrast with the generally intricate
bounded-submodule categories
\cite{Schmidmeier2005Bounded,RingelSchmidmeier2008Invariant}, fixing these
endpoints reduces the relevant orbit problem to left--right equivalence of a
single matrix over a chain ring.

Throughout this section, $V$ is a discrete valuation ring with uniformizer
$\pi$, residue field $k=V/(\pi)$, and
\[
 H_{a,r}=(V/(\pi^a))^r,
 \qquad
 H_{c,s}=(V/(\pi^c))^s,
 \qquad
 m=\min\{a,c\},\quad q=\min\{r,s\}.
\tag{7.1}\label{eq:homocyclic-data}
\]
We assume $a,c,r,s\geq1$ and write $V_h=V/(\pi^h)$.

\begin{lemma}[Extension matrices and endpoint automorphisms]
\label{lem:extension-matrices}
There is a natural identification
\[
 \Ext^1_V(H_{c,s},H_{a,r})
 \cong \operatorname{Mat}_{r\times s}(V_m).
\tag{7.2}\label{eq:matrix-extension-space}
\]
Under this identification, the
$\operatorname{Aut}(H_{a,r})\times\operatorname{Aut}(H_{c,s})$-action is
left--right equivalence:
\[
 X\longmapsto UXW,\qquad
 U\in\operatorname{GL}_r(V_m),\quad
 W\in\operatorname{GL}_s(V_m).
\tag{7.3}\label{eq:left-right-action}
\]
Every pair $(U,W)$ in \eqref{eq:left-right-action} is induced by endpoint
automorphisms.  This is the endpoint action on short exact sequences, rather
than the full middle-term automorphism action used in Hall-polynomial counts;
compare \cite{Schmidmeier2012HallAutomorphisms}.
\end{lemma}

\begin{proof}
The standard resolution of $V/(\pi^c)$ gives
\[
 \Ext^1_V(V/(\pi^c),V/(\pi^a))
 \cong (V/(\pi^a))/\pi^c(V/(\pi^a))\cong V_m.
\]
Additivity of $\Ext^1$ in both variables yields
\eqref{eq:matrix-extension-space}.  Pushout in the first endpoint and
pullback in the second give row and column operations, respectively.  The
reduction maps
\[
 \operatorname{GL}_r(V_a)\twoheadrightarrow\operatorname{GL}_r(V_m),
 \qquad
 \operatorname{GL}_s(V_c)\twoheadrightarrow\operatorname{GL}_s(V_m)
\]
are surjective: lift the entries of an invertible matrix; its determinant
remains a unit because invertibility over every $V_h$ is detected modulo
$\pi$.  Thus the entire left--right action is induced by endpoint
automorphisms.
\end{proof}

\begin{theorem}[Smith orbits and filtered ranks]
\label{thm:smith-orbits}
Every endpoint-automorphism orbit in
$\Ext^1_V(H_{c,s},H_{a,r})$ has a unique Smith profile
\[
 \boldsymbol\nu(X)=(\nu_1,\ldots,\nu_q),
 \qquad
 0\leq\nu_1\leq\cdots\leq\nu_q\leq m,
\tag{7.4}\label{eq:smith-profile}
\]
meaning that $X$ is left--right equivalent over $V_m$ to
\[
 \operatorname{diag}(\pi^{\nu_1},\ldots,\pi^{\nu_q}),
\tag{7.5}\label{eq:smith-diagonal}
\]
with the evident rectangular zero rows or columns and with
$\pi^m=0$ in $V_m$.

The profile has an intrinsic filtered description.  For $1\leq h\leq m$,
let $X_h:V_h^s\to V_h^r$ be the reduction of $X$ and put
\[
 \rho_h(X)=
 \dim_k\bigl(\operatorname{Im}X_h/
 \pi\operatorname{Im}X_h\bigr).
\tag{7.6}\label{eq:valuation-rank-profile}
\]
Then
\[
 \rho_h(X)=\#\{i:\nu_i<h\}.
\tag{7.7}\label{eq:profile-count}
\]
Consequently the nondecreasing sequence
$\boldsymbol\rho(X)=(\rho_1(X),\ldots,\rho_m(X))$ is a complete orbit
invariant.  More precisely,
\[
 \begin{aligned}
 \#\{i:\nu_i=0\}&=\rho_1,\\
 \#\{i:\nu_i=t\}&=\rho_{t+1}-\rho_t
       &&(1\leq t<m),\\
 \#\{i:\nu_i=m\}&=q-\rho_m.
 \end{aligned}
\tag{7.8}\label{eq:recover-smith}
\]
Every nondecreasing sequence
$0\leq\rho_1\leq\cdots\leq\rho_m\leq q$ occurs.
\end{theorem}

\begin{proof}
The ring $V_m$ is a principal ideal chain ring.  Smith reduction over a
principal ideal ring
\cite{Kaplansky1949Elementary,Hungerford1968PIR,
McDonald1974FiniteRings,Newman1971Smith} gives
\eqref{eq:smith-diagonal}; uniqueness of the ideals generated by the diagonal
entries gives \eqref{eq:smith-profile}.  Lemma~\ref{lem:extension-matrices}
then identifies Smith equivalence with the endpoint-automorphism orbits.

Reduce \eqref{eq:smith-diagonal} modulo $\pi^h$.  The image of its $i$th
diagonal coordinate is zero when $\nu_i\geq h$ and is the cyclic ideal
$\pi^{\nu_i}V_h$ when $\nu_i<h$.  In the latter case
\[
 \pi^{\nu_i}V_h/\pi^{\nu_i+1}V_h\cong k.
\]
Thus the minimal number of generators of $\operatorname{Im}X_h$, namely the
dimension in \eqref{eq:valuation-rank-profile}, is the count in
\eqref{eq:profile-count}.  The difference formulas
\eqref{eq:recover-smith} recover every Smith exponent, proving completeness.
Conversely, assign to a prescribed nondecreasing sequence the multiplicities
in \eqref{eq:recover-smith}; the resulting diagonal matrix realizes it.
\end{proof}

For $m=1$, the single number $\rho_1$ is ordinary residue-field rank.  For
$r=s=1$, the profile is a step function: if $X$ has valuation $t<m$, then
$\rho_h(X)=0$ for $h\leq t$ and $\rho_h(X)=1$ for $h>t$; the zero class has
the identically zero profile.  Thus matrix rank and scalar valuation are the
two extreme cases of \eqref{eq:valuation-rank-profile}.

\begin{theorem}[Middle term of a Smith orbit]
\label{thm:smith-middle}
Let
\[
 \varepsilon_X:\quad
 0\longrightarrow H_{a,r}\longrightarrow B_X
 \longrightarrow H_{c,s}\longrightarrow0
\tag{7.9}\label{eq:homocyclic-extension}
\]
represent the matrix $X$, and let $\boldsymbol\nu(X)$ be its Smith profile.
Then
\[
 B_X\cong
 \bigoplus_{i=1}^{q}
 \left(V/(\pi^{\nu_i})\oplus
 V/(\pi^{a+c-\nu_i})\right)
 \oplus (V/(\pi^a))^{r-q}
 \oplus (V/(\pi^c))^{s-q},
\tag{7.10}\label{eq:smith-middle-general}
\]
where $V/(\pi^0)$ is omitted.  In particular, define the positive-exponent
multiset
\[
 \Lambda(X)=
 \bigsqcup_{i=1}^{q}
 \bigl(\{\nu_i:\nu_i>0\}\sqcup\{a+c-\nu_i\}\bigr)
 \sqcup\{a\}^{r-q}\sqcup\{c\}^{s-q}.
\tag{7.11}\label{eq:middle-exponent-multiset}
\]
Then $\Lambda(X)$ is exactly the elementary-divisor multiset of $B_X$.
\end{theorem}

\begin{proof}
Choose generators $u_1,\ldots,u_r$ for the kernel and lifts
$y_1,\ldots,y_s$ of the quotient generators.  If
$X=(\alpha_{ij})$, the middle term has relations
\[
 \pi^au_i=0,\qquad
 \pi^cy_j=\sum_{i=1}^{r}\alpha_{ij}u_i.
\tag{7.12}\label{eq:matrix-presentation}
\]
Endpoint automorphisms do not change the isomorphism type of the middle term,
so Theorem~\ref{thm:smith-orbits} reduces $X$ to
\eqref{eq:smith-diagonal}.  The presentation then separates into $q$ scalar
blocks and the unpaired kernel or quotient generators.  A scalar block with
parameter $\pi^{\nu_i}$ has relation matrix
\[
 \begin{pmatrix}\pi^a&-\pi^{\nu_i}\\0&\pi^c\end{pmatrix}.
\]
Its first Smith divisor is $\pi^{\nu_i}$ and its determinant is
$\pi^{a+c}$, so it contributes
$V/(\pi^{\nu_i})\oplus V/(\pi^{a+c-\nu_i})$.  The unpaired generators
contribute the last two terms of \eqref{eq:smith-middle-general}.
\end{proof}

Formula \eqref{eq:smith-middle-general} turns the orbit classification into
a sharp phase diagram.  Recall that every finite-length primary $C4$-module
is homocyclic, while every finite-length primary $C4^{\ast}$-module is cyclic
or semisimple.

\begin{theorem}[Complete local extension phase diagram]
\label{thm:complete-local-phase}
In \eqref{eq:homocyclic-extension}, the middle term satisfies:
\begin{enumerate}[label=\textup{(\arabic*)}]
\item $B_X$ is $C4$ if and only if exactly one of the following alternatives
holds, or both hold:
\[
 \begin{array}{ll}
 \textup{(Z4)}&a=c\text{ and }X=0,\\
 \textup{(U4)}&r=s\text{ and }X\in\operatorname{GL}_r(V_m).
 \end{array}
\tag{7.13}\label{eq:c4-local-alternatives}
\]
In case \textup{(Z4)}, $B_X\cong(V/(\pi^a))^{r+s}$; in case
\textup{(U4)}, $B_X\cong(V/(\pi^{a+c}))^r$.
\item $B_X$ is $C4^{\ast}$ if and only if exactly one of the following
alternatives holds, or both hold:
\[
 \begin{array}{ll}
 \textup{(Z$\ast$)}&a=c=1\text{ and }X=0,\\
 \textup{(U$\ast$)}&r=s=1\text{ and }X\in V_m^{\times}.
 \end{array}
\tag{7.14}\label{eq:c4star-local-alternatives}
\]
In case \textup{(Z$\ast$)}, $B_X$ is semisimple; in case
\textup{(U$\ast$)}, it is uniserial and cyclic.
\item $B_X$ is strongly $C4^{\ast}$ if and only if it is
$C4^{\ast}$.
\end{enumerate}
Thus \eqref{eq:c4star-local-alternatives} decides every extension between
finite-length primary $C4^{\ast}$ end terms, while
\eqref{eq:c4-local-alternatives} decides every extension between
finite-length primary $C4$ end terms.
\end{theorem}

\begin{proof}
By Theorem~\ref{thm:pid-classification}, $B_X$ is $C4$ exactly when all
members of $\Lambda(X)$ are equal.  Suppose first that some $\nu_i>0$.
The two exponents contributed by the $i$th block can be equal only if
\[
 \nu_i=a+c-\nu_i.
\]
Since $\nu_i\leq\min\{a,c\}$, this forces $a=c=\nu_i=m$.  If instead some
$\nu_j=0$, its block contributes the single exponent $a+c$, which cannot
equal any positive $\nu_i\leq m$.  Hence a homocyclic middle term has either
all $\nu_i=m$ or all $\nu_i=0$.

In the first case $X=0$; its paired factors have exponents $a$ and $c$, so
they, and any unpaired factors, are all equal exactly when $a=c$.  This is
\textup{(Z4)}.  In the second case every Smith entry is a unit.  Any unpaired
factor has exponent $a$ or $c$, smaller than $a+c$, so no unpaired factor may
occur.  Thus $r=s$ and $X$ is invertible over $V_m$, giving
\textup{(U4)}.  The converses follow immediately from
\eqref{eq:smith-middle-general}.  This proves (1).

For (2), the primary classification says that $B_X$ must be cyclic or
semisimple.  If it is cyclic, then its submodule $H_{a,r}$ and quotient
$H_{c,s}$ are cyclic, forcing $r=s=1$; formula
\eqref{eq:smith-middle-general} then shows that the scalar Smith exponent is
zero, equivalently $X\in V_m^{\times}$.  Conversely a unit scalar produces
$V/(\pi^{a+c})$.  If $B_X$ is semisimple, both end terms are semisimple,
so $a=c=1$, and the sequence splits; hence $X=0$.  Conversely the zero
extension with $a=c=1$ has middle term $k^{r+s}$.  Finally, $B_X$ has finite
length, so Proposition~\ref{prop:finite-length-semiweak} proves (3).
\end{proof}

The phase diagram also determines the algebraic geometry of the preservation
sets inside one Yoneda group, including the cases in which no preserving
extension exists.

\begin{corollary}[Exact preservation loci]
\label{cor:locus-geometry}
Set $E=\operatorname{Mat}_{r\times s}(V_m)$ and
\[
 E_4=\{X\in E:B_X\text{ is }C4\},
 \qquad
 E_{\ast}=\{X\in E:B_X\text{ is }C4^{\ast}\}.
\]
With a conditionally displayed set omitted when its condition fails,
\[
 E_4=
 \bigl(\{0\}\ \text{if }a=c\bigr)
 \ \cup\
 \bigl(\operatorname{GL}_r(V_m)\ \text{if }r=s\bigr),
\tag{7.15}\label{eq:c4-locus-general}
\]
and
\[
 E_{\ast}=
 \bigl(\{0\}\ \text{if }a=c=1\bigr)
 \ \cup\
 \bigl(V_m^{\times}\ \text{if }r=s=1\bigr).
\tag{7.16}\label{eq:c4star-locus-general}
\]
Consequently:
\begin{enumerate}[label=\textup{(\arabic*)}]
\item $E_4$ is an additive subgroup of $E$ exactly when either
$a=c$ and $r\neq s$, in which case $E_4=\{0\}$, or
$a=c=r=s=1$, in which case $E_4=E=k$.
\item $E_{\ast}$ is an additive subgroup of $E$ exactly when $a=c=1$.
It equals $E$ for $r=s=1$ and equals $\{0\}$ otherwise.
\item $E_{\ast}$ is nonempty exactly when $a=c=1$ or $r=s=1$; it equals
all of $E$ exactly when $a=c=r=s=1$.
\end{enumerate}
\end{corollary}

\begin{proof}
Equations \eqref{eq:c4-locus-general} and
\eqref{eq:c4star-locus-general} restate
Theorem~\ref{thm:complete-local-phase}.  If $r=s$ but $a\neq c$, then
$E_4=\operatorname{GL}_r(V_m)$ does not contain zero.  If $a=c$ and
$r=s\geq2$, take $U=I_r$ and $W=-I_r+E_{12}$.  Both are invertible, whereas
$U+W=E_{12}$ is nonzero and singular, so
$\{0\}\cup\operatorname{GL}_r(V_m)$ is not additive.  For $r=s=1$ and
$a=c\geq2$, the same failure follows from the units $1$ and $-1+\pi$, whose
sum is the nonzero nonunit $\pi$.  This proves (1).  Formula
\eqref{eq:c4star-locus-general} has only four possibilities:
$\varnothing$, $\{0\}$, $V_m^{\times}$, or $k$.  The unit group is not an
additive subgroup when $m\geq1$ unless adjoining zero gives the field $k$,
which occurs exactly for $a=c=r=s=1$.  Statements (2) and (3) follow.
\end{proof}

\section{Boundary strata and local--global consequences}

We extract the two boundary forms of the filtered Smith invariant.  They are
useful both as explicit phase diagrams and as independent checks on the
general calculation.  Throughout this section $D$ is a principal ideal
domain.  Fix a prime element $p$ of $D$ and integers $a,c\geq1$, and put
$m=\min\{a,c\}$.

\begin{theorem}[Valuation and the middle term]
\label{thm:valuation-middle}
Under the standard identification
\[
\Ext^1_D(D/(p^c),D/(p^a))\cong D/(p^m),
\tag{8.1}\label{eq:ext-identification}
\]
let $\xi$ be an extension class.  Define
\[
s(\xi)=
\begin{cases}
v_p(\xi),&\xi\neq0,\\
m,&\xi=0.
\end{cases}
\]
Thus $0\leq s(\xi)\leq m$, with $s(\xi)<m$ for a nonzero class.  If
\[
\varepsilon_\xi:\quad
0\longrightarrow D/(p^a)\longrightarrow B_\xi
\longrightarrow D/(p^c)\longrightarrow0
\tag{8.2}\label{eq:cyclic-extension}
\]
represents $\xi$, then
\[
B_\xi\cong D/(p^{s(\xi)})\oplus
D/(p^{a+c-s(\xi)}),
\tag{8.3}\label{eq:middle-smith}
\]
where $D/(p^0)$ denotes the zero module.  Moreover, the
$\operatorname{Aut}(D/(p^a))\times\operatorname{Aut}(D/(p^c))$-orbits on
the extension group are precisely the $m+1$ strata
$s=0,1,\ldots,m$.
\end{theorem}

\begin{proof}
Applying $\Hom_D(-,D/(p^a))$ to the standard two-term free resolution of
$D/(p^c)$ gives
\[
\Ext^1_D(D/(p^c),D/(p^a))
\cong (D/(p^a))/p^c(D/(p^a))\cong D/(p^m),
\]
as in the usual computation of cyclic extension groups
\cite{Rotman2009Homological,Weibel1994HomologicalAlgebra}.

Let $x$ be the image in $B_\xi$ of the standard generator of $D/(p^a)$,
and let $y$ lift the standard generator of $D/(p^c)$.  For a representative
$q$ of $\xi$, the middle term has presentation
\[
p^ax=0,\qquad p^cy=qx.
\]
If $\xi\neq0$, write $q=p^su$ with $u$ a unit modulo $p$; for the zero
class take $q=0$ and $s=m$.  The relation matrix
\[
\begin{pmatrix}
p^a&-q\\
0&p^c
\end{pmatrix}
\]
has first determinantal divisor $p^s$ and determinant $p^{a+c}$.
Smith normal form therefore has diagonal entries
$p^s,p^{a+c-s}$, proving \eqref{eq:middle-smith}.

Pushout and pullback along automorphisms act on
\eqref{eq:ext-identification} by multiplication by units.  Conversely,
multiplication by any unit modulo $p^m$ is induced by an automorphism of
$D/(p^a)$.  Two elements of $D/(p^m)$ differ by a unit precisely when they
have the same $p$-adic valuation, with zero as the separate stratum $s=m$.
This proves the orbit statement.
\end{proof}

\begin{theorem}[Complete extension phase diagram]
\label{thm:extension-strata}
In the notation of Theorem~\ref{thm:valuation-middle}, both end terms of
\eqref{eq:cyclic-extension} are strongly $C4^{\ast}$, and:
\begin{enumerate}[label=\textup{(\arabic*)}]
\item $B_\xi$ is $C4$ if and only if either $s(\xi)=0$, or
$\xi=0$ and $a=c$;
\item $B_\xi$ is $C4^{\ast}$ if and only if it is strongly
$C4^{\ast}$, and this occurs if and only if either $s(\xi)=0$, or
$\xi=0$ and $a=c=1$;
\item $B_\xi$ is $C4$ but not $C4^{\ast}$ exactly for the split class
with $a=c\geq2$;
\item every nonzero class divisible by $p$ has a middle term that is not
$C4$.
\end{enumerate}
\end{theorem}

\begin{proof}
The end terms are uniserial, so Theorem~\ref{thm:uniserial} applies.  If
$s(\xi)=0$, formula \eqref{eq:middle-smith} makes $B_\xi$ cyclic and hence
strongly $C4^{\ast}$.  Suppose $s(\xi)>0$.  Then $B_\xi$ has two nonzero
invariant factors with exponents $s=s(\xi)$ and $a+c-s$.  By
Theorem~\ref{thm:pid-classification}, it is $C4$ exactly when these
exponents are equal.  Since $s\leq\min\{a,c\}$, equality forces
$a=c=s$.  This can occur only in the zero stratum $s=m$, and then the
middle term is $(D/(p^a))^2$.  This proves (1).

The same classification says that a two-generator primary module is
$C4^{\ast}$ exactly when it is semisimple.  Here that requires
$s=a+c-s=1$, hence $a=c=1$ and $\xi=0$.  Finite length makes
$C4^{\ast}$ equivalent to strongly $C4^{\ast}$ by
Proposition~\ref{prop:finite-length-semiweak}.  Statements (2)--(4) now
follow.
\end{proof}

\begin{corollary}[Nonlinearity of the preservation locus]
\label{cor:nonlinear-locus}
Identify the extension group in \eqref{eq:ext-identification} with
$E=D/(p^m)$.  Let
\[
E_{\ast}=\{\xi\in E:B_\xi\text{ is }C4^{\ast}\},
\qquad
E_4=\{\xi\in E:B_\xi\text{ is }C4\}.
\]
Then
\[
E_{\ast}=
\begin{cases}
E,&a=c=1,\\
E^{\times},&\text{otherwise},
\end{cases}
\qquad
E_4=
\begin{cases}
E^{\times}\cup\{0\},&a=c,\\
E^{\times},&a\neq c.
\end{cases}
\tag{8.4}\label{eq:preservation-loci}
\]
Here $E^{\times}$ denotes the unit stratum, not a group of extension
automorphisms.  Except in the simple-by-simple case $a=c=1$, neither
preservation locus is an additive subgroup of $E$.
\end{corollary}

\begin{proof}
Formula \eqref{eq:preservation-loci} is Theorem~\ref{thm:extension-strata}.
If $a\neq c$, both loci omit zero and therefore are not subgroups.  If
$a=c\geq2$, take $u=1$ and $v=-1+p$.  Both are units modulo $p^m$, while
$u+v=p$ has valuation one.  Thus $u,v\in E^{\times}$ but
$u+v\notin E^{\times}\cup\{0\}$.  When $a=c=1$, every middle term is
either cyclic of length two or semisimple, so both loci equal $E$.
\end{proof}

The preceding theorem globalizes without loss of information.

\begin{corollary}[Local--global extension criterion]
\label{cor:local-global}
Let $A=D/(\alpha)$ and $C=D/(\gamma)$ be finite-length cyclic torsion
modules, and write
\[
a_p=v_p(\alpha),\qquad c_p=v_p(\gamma).
\]
An extension class $\xi\in\Ext^1_D(C,A)$ has local coordinates
$\xi_p\in D/(p^{\min(a_p,c_p)})$ for the primes dividing both $\alpha$
and $\gamma$.  Its middle term $B_\xi$ is strongly $C4^{\ast}$ (equivalently
$C4^{\ast}$) if and only if, for every such prime $p$, either
\[
v_p(\xi_p)=0,
\qquad\text{or}\qquad
\xi_p=0\ \text{ and }\ a_p=c_p=1.
\tag{8.5}\label{eq:local-strong}
\]
It is $C4$ if and only if, for every common prime, either
\[
v_p(\xi_p)=0,
\qquad\text{or}\qquad
\xi_p=0\ \text{ and }\ a_p=c_p.
\tag{8.6}\label{eq:local-c4}
\]
\end{corollary}

\begin{proof}
Primary decomposition splits the extension into its $p$-primary parts;
for distinct primes the cross-$\Ext^1$ groups vanish.  A prime occurring in
only one end term contributes that same cyclic primary component to the
middle term.  Lemma~\ref{lem:primary-separation} and
Theorem~\ref{thm:extension-strata} therefore give
\eqref{eq:local-strong} and \eqref{eq:local-c4} componentwise.
\end{proof}

\begin{corollary}[The unit extension family]\label{thm:pid-family}
For $m,n\geq1$, the sequence
\[
0\longrightarrow D/(p^m)
\xrightarrow{\ \bar a\mapsto\overline{p^n a}\ }
D/(p^{m+n})
\longrightarrow D/(p^n)\longrightarrow0
\tag{8.7}\label{eq:pid-extension}
\]
represents the unit orbit in
$\Ext^1_D(D/(p^n),D/(p^m))$.  It is nonsplit, and all three terms are
strongly $C4^{\ast}$.
\end{corollary}

\begin{proof}
With $x=\overline{p^n}$ and $y=\overline1$ in the middle term, the relation
$p^ny=x$ shows that the extension parameter is $1$.  Thus its valuation is
zero, and Theorem~\ref{thm:extension-strata} applies.  It is nonsplit
because its class is nonzero.
\end{proof}

\begin{example}[Three phases in one extension group]
\label{ex:three-phases}
Take $a=c=2$.  The three automorphism orbits in
$\Ext^1_D(D/(p^2),D/(p^2))\cong D/(p^2)$ have sharply different behavior:
\[
\begin{array}{c|c|c}
\text{class stratum}&\text{middle term}&\text{status}\\ \hline
v_p(\xi)=0&D/(p^4)&\text{strongly }C4^{\ast}\\
v_p(\xi)=1&D/(p)\oplus D/(p^3)&\text{not }C4\\
\xi=0&(D/(p^2))^2&C4\text{ but not }C4^{\ast}.
\end{array}
\]
Thus neither splitting nor nonsplitting governs preservation: the decisive
invariant is the valuation orbit of the extension class.
\end{example}

The next result replaces scalar valuation by matrix rank and shows that the
orbit method is not confined to cyclic end terms.  Put $k=D/(p)$.

\begin{theorem}[Matrix-rank strata for semisimple end terms]
\label{thm:matrix-rank-strata}
Choose bases to identify
\[
\Ext^1_D(k^s,k^r)\cong\operatorname{Mat}_{r\times s}(k).
\tag{8.8}\label{eq:matrix-ext}
\]
Let $\Xi$ be the matrix of an extension
\[
0\longrightarrow k^r\longrightarrow B_\Xi
\longrightarrow k^s\longrightarrow0
\]
and put $\rho=\operatorname{rank}_k\Xi$.  Then
\[
B_\Xi\cong
(D/(p^2))^{\rho}\oplus k^{\,r+s-2\rho}.
\tag{8.9}\label{eq:rank-middle}
\]
The $\operatorname{GL}_r(k)\times\operatorname{GL}_s(k)$-orbits are
precisely the rank strata $0\leq\rho\leq\min\{r,s\}$.  Moreover:
\begin{enumerate}[label=\textup{(\arabic*)}]
\item $B_\Xi$ is $C4$ if and only if $\rho=0$ or
$\rho=r=s$;
\item $B_\Xi$ is $C4^{\ast}$, equivalently strongly $C4^{\ast}$, if and
only if $\rho=0$ or $\rho=r=s=1$;
\item in particular, a full-rank square class of size at least two has a
$C4$ middle term that is not $C4^{\ast}$, while every nonzero
non-full-rank class has a middle term that is not $C4$.
\end{enumerate}
\end{theorem}

\begin{proof}
Additivity of $\Ext^1$ in each variable gives
\eqref{eq:matrix-ext}.  Changing bases in the submodule and quotient acts
by invertible row and column operations, so the orbit of $\Xi$ is
determined by its rank.  Reduce $\Xi$ to
$\operatorname{diag}(I_\rho,0)$.  Each nonzero diagonal entry represents
the unique nonzero orbit in $\Ext^1_D(k,k)$ and contributes the nonsplit
middle term $D/(p^2)$.  The remaining $r-\rho$ copies from the submodule
and $s-\rho$ copies from the quotient split and contribute copies of $k$.
This proves \eqref{eq:rank-middle}.

The invariant exponents of $B_\Xi$ are therefore $2$, repeated $\rho$
times, and $1$, repeated $r+s-2\rho$ times.  By
Theorem~\ref{thm:pid-classification}, they are all equal exactly when
$\rho=0$ or $r+s-2\rho=0$; the latter equality, together with
$\rho\leq\min\{r,s\}$, is equivalent to $\rho=r=s$.  This proves (1).
For $C4^{\ast}$, the same theorem requires the primary module to be
semisimple or cyclic.  These alternatives give respectively $\rho=0$ and
$\rho=r=s=1$.  Proposition~\ref{prop:finite-length-semiweak} supplies the
strong equivalence, proving (2) and (3).
\end{proof}

\section{Local--global completeness over Dedekind domains}

The Smith theorem closes the local orbit problem.  We now assemble those
local answers and obtain a criterion for arbitrary finite-length torsion
$C4^{\ast}$ end terms over a Dedekind domain, including mixed cyclic and
semisimple primary components.

Let $D$ be a Dedekind domain and let $A,C$ be finite-length torsion
$C4^{\ast}$-modules.  For every prime ideal $\mathfrak p$ in the common
support $\Sigma=\operatorname{Supp}(A)\cap\operatorname{Supp}(C)$, write
after localization
\[
 A_{(\mathfrak p)}\cong
 (D_{\mathfrak p}/(\pi_{\mathfrak p}^{a_{\mathfrak p}}))^{r_{\mathfrak p}},
 \qquad
 C_{(\mathfrak p)}\cong
 (D_{\mathfrak p}/(\pi_{\mathfrak p}^{c_{\mathfrak p}}))^{s_{\mathfrak p}}.
\tag{9.1}\label{eq:global-endpoint-data}
\]
Theorem~\ref{thm:pid-classification} says that
\[
 r_{\mathfrak p}=1\text{ or }a_{\mathfrak p}=1,
 \qquad
 s_{\mathfrak p}=1\text{ or }c_{\mathfrak p}=1.
\tag{9.2}\label{eq:c4star-endpoint-restriction}
\]
Put $m_{\mathfrak p}=\min\{a_{\mathfrak p},c_{\mathfrak p}\}$.  An
extension class $\xi\in\Ext^1_D(C,A)$ has a local matrix coordinate
\[
 X_{\mathfrak p}\in
 \operatorname{Mat}_{r_{\mathfrak p}\times s_{\mathfrak p}}
 \bigl(D_{\mathfrak p}/(\pi_{\mathfrak p}^{m_{\mathfrak p}})\bigr)
\tag{9.3}\label{eq:global-local-matrix}
\]
at every $\mathfrak p\in\Sigma$.

\begin{theorem}[Complete local--global phase diagram]
\label{thm:global-phase}
Let
\[
 0\longrightarrow A\longrightarrow B_\xi\longrightarrow C
 \longrightarrow0
\tag{9.4}\label{eq:global-extension}
\]
represent $\xi$.  Then:
\begin{enumerate}[label=\textup{(\arabic*)}]
\item $B_\xi$ is $C4$ if and only if, for every
$\mathfrak p\in\Sigma$, at least one of the following holds:
\[
 \begin{aligned}
 &a_{\mathfrak p}=c_{\mathfrak p}
 &&\text{and }X_{\mathfrak p}=0;\\
 &r_{\mathfrak p}=s_{\mathfrak p}
 &&\text{and }X_{\mathfrak p}\text{ is invertible over }
 D_{\mathfrak p}/(\pi_{\mathfrak p}^{m_{\mathfrak p}}).
 \end{aligned}
\tag{9.5}\label{eq:global-c4-criterion}
\]
\item $B_\xi$ is $C4^{\ast}$, equivalently strongly $C4^{\ast}$, if
and only if, for every $\mathfrak p\in\Sigma$, at least one of the
following holds:
\[
 \begin{aligned}
 &a_{\mathfrak p}=c_{\mathfrak p}=1
 &&\text{and }X_{\mathfrak p}=0;\\
 &r_{\mathfrak p}=s_{\mathfrak p}=1
 &&\text{and }X_{\mathfrak p}\text{ is a unit.}
 \end{aligned}
\tag{9.6}\label{eq:global-c4star-criterion}
\]
\item The
$\operatorname{Aut}(A)\times\operatorname{Aut}(C)$-orbit of $\xi$ is
determined by the finite family of local valuation--rank profiles
$\{\boldsymbol\rho(X_{\mathfrak p}):\mathfrak p\in\Sigma\}$.
\end{enumerate}
\end{theorem}

\begin{proof}
Primary decomposition gives
\[
 \Ext^1_D(C,A)\cong
 \bigoplus_{\mathfrak p\in\Sigma}
 \Ext^1_{D_{\mathfrak p}}
 (C_{(\mathfrak p)},A_{(\mathfrak p)}),
\tag{9.7}\label{eq:ext-primary-product}
\]
because the cross-$\Ext^1$ groups for distinct primes vanish.  The middle
term decomposes into the corresponding primary middle terms.  At a prime
outside $\Sigma$, its primary component is the unique nonzero endpoint
component and is already $C4^{\ast}$.  At a common prime,
Theorems~\ref{thm:smith-orbits} and
\ref{thm:complete-local-phase} give respectively the orbit invariant and
the two criteria.  Primary separation
(Lemma~\ref{lem:primary-separation}) assembles the local conclusions.
Finally, every middle term has finite length, so strongly $C4^{\ast}$ and
$C4^{\ast}$ coincide.
\end{proof}

The next consequences show that the global criterion is not merely a
decision algorithm: it identifies exactly when preservation is possible,
automatic, or compatible with Baer addition.

\begin{corollary}[Existence, universality and the mixed no-go zone]
\label{cor:global-existence}
In the setting of Theorem~\ref{thm:global-phase}:
\begin{enumerate}[label=\textup{(\arabic*)}]
\item There exists a class $\xi$ with $B_\xi$ strongly $C4^{\ast}$ if and
only if, for every $\mathfrak p\in\Sigma$,
\[
 a_{\mathfrak p}=c_{\mathfrak p}=1
 \quad\text{or}\quad
 r_{\mathfrak p}=s_{\mathfrak p}=1.
\tag{9.8}\label{eq:existence-criterion}
\]
\item Every extension class has strongly $C4^{\ast}$ middle term if and
only if
\[
 A_{(\mathfrak p)}\cong D/\mathfrak p
 \cong C_{(\mathfrak p)}
 \qquad(\mathfrak p\in\Sigma).
\tag{9.9}\label{eq:universal-preservation}
\]
\item The split extension has $C4^{\ast}$ middle term if and only if both
common primary components are semisimple at every
$\mathfrak p\in\Sigma$.
\item In particular, if at some common prime one endpoint component is a
nonsimple cyclic module and the other is a nonsimple semisimple module,
then no extension between $A$ and $C$ has even a $C4$ middle term.
\end{enumerate}
\end{corollary}

\begin{proof}
Locally, Corollary~\ref{cor:locus-geometry} says that the
$C4^{\ast}$-preservation locus is nonempty exactly under one of the two
conditions in \eqref{eq:existence-criterion}, and that it is the entire
extension group exactly in the simple-by-simple case.  It contains zero
exactly when $a_{\mathfrak p}=c_{\mathfrak p}=1$.  These observations prove
(1)--(3) after taking the product in \eqref{eq:ext-primary-product}.

For (4), suppose for example that
$A_{(\mathfrak p)}\cong(D/\mathfrak p)^r$ with $r\geq2$ and
$C_{(\mathfrak p)}\cong D/\mathfrak p^c$ with $c\geq2$; the reverse case is
symmetric.  Here neither endpoint exponents nor endpoint ranks agree, so
both alternatives in \eqref{eq:global-c4-criterion} fail for every local
matrix.  Hence no local, and therefore no global, middle term is $C4$.
\end{proof}

\begin{corollary}[Additivity criterion]
\label{cor:global-additivity}
Let
\[
 \mathcal P_{\ast}(A,C)=
 \{\xi\in\Ext^1_D(C,A):B_\xi\text{ is }C4^{\ast}\}.
\]
Then $\mathcal P_{\ast}(A,C)$ is an additive subgroup of
$\Ext^1_D(C,A)$ if and only if, at every prime in the common support, both
$A_{(\mathfrak p)}$ and $C_{(\mathfrak p)}$ are semisimple.  In that case
each local factor of the preservation locus is either the whole
residue-field extension space in the simple-by-simple case or the zero
subgroup otherwise.
\end{corollary}

\begin{proof}
The global preservation locus is the product of the local sets
\eqref{eq:c4star-locus-general}.  By Corollary~\ref{cor:locus-geometry}, a
nonempty local set is additive exactly when the two local exponents are one.
The assertion follows from \eqref{eq:ext-primary-product}.
\end{proof}

Thus the obstruction to an exact structure obtained by selecting all
property-preserving Yoneda classes is now complete in finite length: cyclic
nonsimple overlap produces a unit stratum, mixed higher-rank overlap can
produce the empty set, and semisimple overlap is the unique additive regime.
Here ``exact structure'' is used in the standard Quillen--Yoneda sense; see
\cite{Buehler2010Exact}.

\section{Conclusion}

The exact-sequence behavior of $C4^{\ast}$-modules is asymmetric for a
simple reason: the property is hereditary by definition but not
cohereditary.  Kernels are therefore automatic.  Extensions and cokernels
are not, and the polynomial examples in Section~3 exhibit both
failures with no appeal to abstract obstruction schemes.

The principal positive result is the Smith-orbit theorem.  Between any two
homocyclic primary $C4$-modules, endpoint-automorphism orbits in $\Ext^1$
are classified by a valuation--rank profile over a truncated valuation ring.
The profile determines every elementary divisor of the middle term.  The
$C4$ locus consequently collapses to the equal-exponent zero orbit and the
equal-rank invertible orbit, while the $C4^{\ast}$ locus collapses further to
the semisimple zero orbit and the cyclic unit orbit.  Scalar valuation and
residue-field rank are not separate phenomena but the two boundary forms of
this filtered invariant.

Primary decomposition globalizes the local theorem to arbitrary
finite-length torsion $C4^{\ast}$ end terms over a Dedekind domain.  It also
reveals a rigid trichotomy unavailable from closure arguments alone: the
preservation locus may be the full local extension space, its zero subgroup,
a nonadditive unit stratum, or the empty set.  In particular, the mixed
nonsimple cyclic--semisimple configuration admits no $C4$ middle term at all.
Thus the exact-sequence obstruction is determined at the level of complete
automorphism orbits, not merely by isolated counterexamples.

\section*{Acknowledgements}

The author thanks the Department of Mathematics, Government College
(Autonomous), Rajahmundry, for institutional support.

\bibliographystyle{amsplain}
\bibliography{references}

\end{document}